\documentclass[a4paper,12pt]{article}
\usepackage{amssymb}
\usepackage{amsmath}
\usepackage{amsthm}
\usepackage{mathrsfs}
\usepackage{amsfonts}
\usepackage{color}

\title{Variations on known and recent cardinality bounds}
\date{}
\author{F.A. Basile\footnote{University of Messina}, M. Bonanzinga\footnote{University of Messina}, N. Carlson\footnote{California Lutheran University}}

\begin{document}
\maketitle

\newtheorem {theorem}{Theorem}[section]
\newtheorem {lemma}{Lemma}[section]
\newtheorem {cor}{Corollary}[section]
\newtheorem {defin}{Definition}[section]
\newtheorem{remark}{Remark}[section]
\newtheorem{claim}{Claim}[section]
\newtheorem{es}{Example}[section]
\newtheorem{quest}{Question}[section]
\newtheorem{prop}{Proposition}[section]
\newcommand{\my}[1]{\textcolor{red}{\sf #1}}

\begin{abstract}
Sapirovskii \cite{SAP} proved that  $|X|\leq\pi\chi(X)^{c(X)\psi(X)}$, for a regular space $X$. We introduce the
$\theta$-\emph{pseudocharacter} of a Urysohn space $X$, denoted by $\psi_\theta (X)$, and prove that the previous inequality holds for Urysohn spaces replacing 
the bounds on celluarity $c(X)\leq\kappa$ and on pseudocharacter $\psi(X)\leq\kappa$ with a bound on Urysohn cellularity
$Uc(X)\leq\kappa$ (which is a weaker conditon because $Uc(X)\leq c(X)$) and on $\theta$-pseudocharacter $\psi_\theta (X)\leq\kappa$ respectivly (Note that in general $\psi(\cdot)\leq\psi_\theta (\cdot)$ and in the class of regular spaces $\psi(\cdot)=\psi_\theta(\cdot)$).
Further, in \cite{BSS} the authors generalized the Dissanayake and Willard's inequality:
$|X|\leq 2^{aL_{c}(X)\chi(X)}$, for Hausdorff spaces $X$ \cite{WD}, in the class of $n$-Hausdorff spaces and de Groot's result: $|X|\leq 2^{hL(X)}$, for Hausdorff spaces \cite{dG}, in the class of $T_1$ spaces (see Theorems 2.22 and 2.23 in \cite{BSS}). In this paper we restate  Theorem 2.22  in \cite{BSS} in the class of $n$-Urysohn spaces and give a variation of
Theorem 2.23  in \cite{BSS} using new cardinal functions, denoted by $UW(X)$, $\psi w_\theta(X)$, $\theta\hbox{-}aL(X)$, $h\theta\hbox{-}aL(X)$, $\theta\hbox{-}aL_c(X)$ and $\theta\hbox{-}aL_{\theta}(X)$.
In \cite{BCCS} the authors introduced the \emph{Hausdorff point separating weight of a space} $X$ denoted by $Hpsw(X)$ and 
proved a Hausdorff version of Charlesworth's inequality $|X|\leq psw(X)^{L(X)\psi(X)}$ \cite{C}. In this paper, we introduce the \emph{Urysohn point separating weight of a space} $X$, denoted by $Upsw(X)$, and prove that $|X|\leq Upsw(X)^{\theta\hbox{-}aL_{c}(X)\psi(X)}$, for a Urysohn space $X$.

\end{abstract}



\section{Introduction}

We shall follow notations from \cite{ENG} and \cite{H}.
Recall that a space $X$ is \textit{Urysohn} if for every two distinct points $x,\;y\in X$ there are open sets $U$ and $V$ such that $x\in U$, $y\in V$ and $\overline{U}\cap\overline{V}=\emptyset$. 

For a space $X$, we denote by $\chi(X)$ (resp., $\psi(X)$, $\pi\chi(X), c(X), t(X))$ the \emph{character}, (resp., \emph{pseudocharacter}, $\pi$\emph{-character}, \emph{celluarity}, \emph{tightness}) of a space $X$ \cite{ENG}.

The $\theta$\textit{-closure} of a set $A$ in a space $X$ is the set $cl_{\theta}(A)=\{x\in X:$ for every neighborhood $U\ni x , \overline{U}\cap A\neq\emptyset\}$; $A$ is said to be $\theta$-closed if $A = cl_\theta(A)$ \cite{VEL}. Considering the fact that the $\theta$-closure operator is not in general idempotent, Bella and Cammaroto defined in \cite{BC} the $\theta$\textit{-closed hull} of a subset $A$ of a space $X$, denoted by $[A]_{\theta}$, that is the smallest $\theta$-closed subset of $X$ containing $A$.  The $\theta$\textit{-tightness} of $X$ at $x\in X$ is $t_{\theta}(x,X)=\min \{k :$ for every $A\subseteq X$ with $x\in cl_{\theta}(A)$ there exists $B\subseteq A$ such that $|B|\leq k$ and $x\in cl_{\theta}(B)$; the $\theta$\textit{-tightness of} $X$ is $t_{\theta}(X)=sup\{t_{\theta}(x,X):\;x\in X\}$ \cite{CK}. We have that tightness and $\theta$-tightness are independent (see Example 11 and Example 12 in \cite{PCC}), but if $X$ is a regular space then $t(X)=t_{\theta}(X)$. The $\theta$-density of $X$ is $d_{\theta}(X)=min\{k:\;A\subseteq X\;\hbox{, }A\;\hbox{is a dense subset of }X\hbox{and }|A|\leq k\}$. We say that a subset $A$ of $X$ is $\theta$\textit{-dense} in $X$ if $cl_{\theta}(A)=X$.

If $X$ is a Hausdorff space, the \textit{closed pseudocharacter of a point} $x$ in $X$ is $\psi_{c}(x,X)=\min\{|{\cal U}| : {\cal U}$ is a family of open neighborhoods of $x$ and $\left\{x\right\}$ is the intersection of the closure of $\cal U\}$; the \textit{closed pseudocharacter of} $X$ is $\psi_{c}(X)=sup\{\psi_{c}(x,X):\;x\in X\}$ (see \cite{SH} where it is called $S\psi(X)$). The \textit{Urysohn pseudocharacter of} $X$, denoted by $U\psi(X)$, is the smallest cardinal $k$ such that for each point $x\in X$ there is a collection $\{V(\alpha,x):\;\alpha<k\}$ of open neighborhoods of $x$ such that if $x\neq y$, then there exist $\alpha,\;\beta<k$ such that $\overline{V(\alpha,x)}\cap\overline{V(\beta,y)}=\emptyset$ \cite{STA}; this cardinal function is defined only for Urysohn spaces.
The \emph{Urysohn-cellularity} of a space X is $Uc(X) =\ sup\{|{\cal V}| : {\cal V}$ is Urysohn-cellular$\}$ (a collection $\cal V$ of open subsets of $X$ is called \emph{Urysohn-cellular}, if $O_1,\; O_2$ in $\cal V$ and $O_1\neq O_2$ implies
$\overline{O_1}\cap\overline{O_2}=\emptyset$). Of course, $Uc(X) \leq c(X)$.

The \textit{almost Lindel\"of degree} of a subset $Y$ of a space $X$ is $aL(Y,X)=\min \{k :$ for every cover $\mathcal{V}$ of $Y$ consisting of open subsets of $X$, there exists $\mathcal{V'}\subseteq\mathcal{V}$ such that $|\mathcal{V'}|\leq k$ and
$\bigcup\{\overline{V}:\;V\in\mathcal{V'}\}=Y$\}.  The function $aL(X,X)$ is called the \textit{almost Lindel\"of degree} of $X$ and denoted by $aL(X)$ (see \cite{WD} and \cite{Hod}). The \textit{almost  Lindel\"of degree of $X$ with respect to closed subsets of $X$} is $aL_{c}(X)=\sup\{aL(C,X):\;C\subseteq X\;is\;closed\}$.

For a subset A of a space X we will denote by $[A]^{\leq\lambda}$ the
family of all subsets of A of cardinality $\leq\lambda$.

Sapirovskii \cite{SAP} proved that  $|X|\leq\pi\chi(X)^{c(X)\psi(X)}$, for a regular space $X$. Later Shu-Hao \cite{SH} proved that the previous inequality holds in the class of Hausdorff spaces by replacing the pseudocharacter with the closed pseudocharacter. In Section \ref{Section1}  we introduce  the
$\theta$-\emph{pseudocharacter} of a Urysohn space $X$, denoted by $\psi_\theta (X)$ and prove the following result:

$\bullet$ $|X|\leq\pi\chi(X)^{Uc(X)\psi_{\theta}(X)}$ for a Urysohn space $X$.

\noindent A space $X$ is $n$-\textit{Urysohn} \cite{BCM} (resp.  $n$-\textit{Hausdorff} \cite{B}), $n\in \omega$,  if for every $x_{1},x_{2},...,x_{n}\in X$ there exist open subsets $U_{1},U_{2},...,U_{n}$ of $X$ such that $x_{1}\in U_{1},\;x_{2}\in U_{2},..., x_{n}\in U_{n}$ and 
$\bigcap_{i=1}^{n}\overline{U_{i}}=\emptyset$ (resp, $\bigcap_{i=1}^{n}U_{i}=\emptyset$).
In \cite{BSS} the authors generalized the Dissanayake and Willard's inequality:
$|X|\leq 2^{aL_{c}(X)\chi(X)}$, for Hausdorff spaces $X$ \cite{WD}, in the class of $n$-Hausdorff spaces and de Groot's result: $|X|\leq 2^{hL(X)}$, for Hausdorff spaces \cite{dG}, in the class of $T_1$ spaces. In particular, they used two new cardinal functions, denoted by $HW(X)$, $\psi w(X)$, to obtain the following results:

$\bullet$ If $X$ is a $T_{1}$ $n$-Hausdorff ($n\in\omega$) space, then $|X|\leq HW(X)2^{aL_{c}(X)\chi(X)}$.

$\bullet$ If $X$ is a $T_{1}$ space, then $|X|\leq HW(X)\psi w(X)^{haL(X)}$.

\noindent In Section \ref{Section2} we introduce new cardinal functions, denoted by $UW(X)$, $\psi w_\theta(X)$, $\theta\hbox{-}aL(X)$, $h\theta\hbox{-} aL(X)$, $\theta\hbox{-}aL_c(X)$ and $\theta\hbox{-}aL_{\theta}(X)$ such that $HW(X)\leq UW(X)$, $\psi w(X)\leq\psi w_\theta(X)$ and $\theta\hbox{-}aL(X)\leq aL(X)$, restate  Theorem 2.22  in \cite{BSS} in the class of $n$-Urysohn spaces and give a variation of
Theorem 2.23  in \cite{BSS}.
In particular, we prove the following results:

$\bullet$ If $X$ is a $T_{1}$ $n$-Urysohn ($n\in\omega$)  space, then $|X|\leq UW(X)2^{\theta\hbox{-}aL_{\theta}(X)\chi(X)}$.

$\bullet$ If $X$ is a $T_{1}$ space then $|X|\leq UW(X)\psi w_{\theta}(X)^{h\theta\hbox{-}aL(X)}$.

\noindent In \cite{BCCS} the authors introduced the \emph{Hausdorff point separating weight of a space} $X$ denoted by $Hpsw(X)$ and 
proved a Hausdorff version of Charlesworth's inequality $|X|\leq psw(X)^{L(X)\psi(X)}$ \cite{C}. In a similar way, in Section \ref{section3} we introduce \emph{Urysohn point separating weight of a space} $X$, denoted by $Hpsw(X)$, and prove the following result:

$\bullet$ If $X$ is a Urysohn space, then $|X|\leq Upsw(X)^{\theta\hbox{-}aL_{c}(X)\psi(X)}$.

\section{A generalization of Sapirovskii's inequality $|X|\leq\pi\chi(X)^{c(X)\psi(X)}$.}\label{Section1}

\begin{defin}\label{theta-pseudocharacter}
\textnormal{If $X$ is a Urysohn space, we define $\theta$\textit{-pseudocharacter of a point x}$ \in X$ the smallest cardinal $k$ such that $\left\{x\right\}$ is the intersection of the $\theta$-closure of the closure of a family of open neighborhood of $x$ having cardinality less or equal to $k$; we denote it with $\psi_{\theta}(x,X)$.} \textnormal{The $\theta$\textit{-pseudocharacter of} $X$ is:}
$$\psi_{\theta}(X)=sup\{\psi_{\theta}(x,X):\;x\in X\}.$$
\end{defin}

The following result is trivial:

\begin{prop}\label{prop1}
$X$ is a Urysohn space iff for every $x \in X$, $\{x\}$ is the intersection of the $\theta$-closure of the closure of a family of open neighborood of $x$.
\end{prop}
\proof

\noindent Let $X$ be a Urysohn space and $x\in X$. For every $y\in X\setminus\{x\}$, there exist $U_{y}$ and $V_{y}$ open disjoint subsets of $X$ such that $x\in U_{y}$, $y\in V_{y}$ and $\overline{U_{y}}\cap\overline{V_{y}}=\emptyset$. So, y $\notin cl_{\theta}(\overline{U_{y}})$ and $\{x\}=\bigcap_{y\in X\setminus\{x\}}cl_{\theta}(\overline{U_{y}})$. Viceversa let $x,\;y$ be distinct points of $X$. By hypothesis there exists an open neighbourhood $V$ of $x$ such that $y\notin cl_{\theta}(\overline{V})$.Then there exists an open subset $U$ of $X$ such that $y\in U$ and $\overline{U}\cap \overline{V}=\emptyset$. So $X$ is Urysohn.
\endproof

We have that:
$$\psi(X)\leq\psi_{c}(X)\leq\psi_{\theta}(X)\leq U\psi(X)\leq\chi(X).$$

Since for a regular space $X$, $cl_{\theta}(A)=\overline{A}$ for every $A\subseteq X$ \cite{Ency}, we have that for a regular space $X$, $\psi_{c}(X)=\psi_{\theta}(X)$. In general this need not be true for non regular spaces. Indeed if we consider $\mathbb{R}$ with the countable complement topology we have that $\overline{\mathbb{Q}}\neq cl_{\theta}(\mathbb{Q})$.
 
\begin{quest} \rm
Is there a Urysohn space such that $\psi_{c}(X)<\psi_{\theta}(X)$?
\end{quest}

It was proved in \cite{BC} that for Urysohn spaces, $|cl_{\theta}(A)|\leq |A|^{\chi(X)}$ for every $A\subseteq X$
and further this inequality was used
for the estimation of cardinality of Lindel\"{o}f spaces. Since 
$t_{\theta}(X)\psi_{\theta}(X)\leq\chi(X)$, the following proposition improves the result in \cite{BC}. (Note that if $X=\omega\cup \{p\}$, with $p\in\omega^*$,
we have that $\aleph_{0}=t_{\theta}(X)\psi_{\theta}(X)<\chi(X)$.)

\begin{prop}\label{P}
Let $X$ be a Urysohn space such that $t_{\theta}(X)\psi_{\theta}(X)\leq k$. Then for every $A\subseteq X$ we have that $|cl_{\theta}(A)|\leq |A|^{k}$.
\end{prop}
\proof
Let $x\in cl_{\theta}(A)$, since $\psi_{\theta}(X)\leq k$ there exist a family $\{U_{\alpha}(x)\}_{\alpha<k}$ of neighborhood of $x$ such that $\{x\}=\bigcap_{\alpha<k}cl_{\theta}(\overline{U_{\alpha}(x)})$. We want to prove that $x\in cl_{\theta}(\overline{U_{\alpha}(x)}\cap A)$, $\forall \alpha<k$. Let $U$ be a neighborhood of $x$ and $\alpha < k$. Then 
$\emptyset\neq \overline{U\cap U_{\alpha}(x)}\cap A \subseteq \overline{U}\cap \overline{U_{\alpha}(x)}\cap A$. This shows that $x\in cl_{\theta}(\overline{U_{\alpha}(x)}\cap A)$. Since  $t_{\theta}(X)\leq k$, there exists $A_{\alpha}\subset \overline{U_{\alpha}(x)}\cap A$ such that $|A_{\alpha}|\leq k$ and $x\in cl_{\theta}(A_{\alpha})$. Then $\{x\}=\bigcap_{\alpha<k}cl_{\theta}(A_{\alpha})$ and $\{A_{\alpha}\}_{\alpha<k}\in[[A]^{\leq k}]^{\leq k}$, so $|cl_{\theta}(A)|\leq |[[A]^{\leq k}]^{\leq k}|=|A|^{k}$.
\endproof

\begin{cor}\cite{BC}
If $X$ is a Urysohn space then for every $A\subseteq X$ we have that $|cl_{\theta}(A)|\leq |A|^{\chi(X)}$.
\end{cor}

The following result is the analogue of 2.20 in \cite{J} in the case of Urysohn spaces.
 
\begin{cor}\label{C}
If $X$ is a Urysohn space then $|X|\leq d_{\theta}(X)^{t_{\theta}(X)\psi_{\theta}(X)}$.
\end{cor}
\proof
If $A$ is $\theta$-dense subset of $X$, i.e. $cl_{\theta}(A)=X$, we have that $|A|\leq d_{\theta}(X)$ and from the above theorem we have that $|cl_{\theta}(A)|\leq |A|^{t_{\theta}(X)\psi_{\theta}(X)}$, so $|X|\leq d_{\theta}(X)^{t_{\theta}(X)\psi_{\theta}(X)}$.
\endproof

\noindent The authors know that I. Gotchev obtained independently the results given in Proposition \ref{P} and Corollary \ref{C}.

Now we prove the following result:

\begin{lemma}\label{lemma}

Let $X$ be a topological space, $\mathcal{B}$ a $\pi$-base for $X$ and $\mathcal{W}$ a family of open sets.
Let  $\mathcal{M}$ be a maximal Urysohn cellular subfamily  of $\{U\in\mathcal{B}:\;U\subseteq W\;for\;some\;W\in\mathcal{W}\}$.
Then $cl_{\theta}\left(\bigcup\overline{\mathcal{M}}\right)\supseteq\bigcup\;\mathcal{W}$.
\end{lemma}
\begin{proof}
Using Zorn's Lemma we can say that there exists a maximal Urysohn-cellular subfamily $\mathcal{M}$ of $\{U\in\mathcal{B}:\;U\subseteq W\;for\;some\;W\in\mathcal{W}\}$.
We want to prove that $cl_{\theta}\left(\bigcup\;\overline{\mathcal{M}}\right)\supseteq \bigcup\;\mathcal{W}$. 
Assume, by the way of contradiction, that $cl_{\theta}\left(\bigcup\;\overline{\mathcal{M}}\right) \not\supset \bigcup\;\mathcal{W}$. Let $x\in \bigcup\mathcal{W}$ such that $x\notin cl_{\theta}(\bigcup\;\overline{\mathcal{M}})$. Then there exists an open set $U$ such that $x\in U$ such that $\overline{U}\cap \overline{M}=\emptyset$, $\forall M\in\mathcal{M}$. 
So $x\notin M$, $\forall M\in\mathcal{M}$. Let $W\in\mathcal{W}$ such that $x\in W$. $\mathcal{M}\cup \{U\cap W\}$ is a Urysohn cellular family.
Since $\mathcal{B}$ is a $\pi$-base for $X$ and $U\cap W$ is an open set containing $x$, there exists $B\in\mathcal{B}$ such that $B\subseteq U\cap W$, so $\mathcal{M'}=\mathcal{M}\cup \{B\}$ is a Urysohn cellular subfamily of $\{U\in\mathcal{B}:\;U\subseteq W\;for\;some\;W\in\mathcal{W}\}$ containing $\mathcal{M}$; a contradiction.
\end{proof}

\begin{theorem}\label{theta-theorem}
Let $X$ be a Urysohn space. Then $|X|\leq\pi\chi(X)^{Uc(X)\psi_{\theta}(X)}$.
\end{theorem}
\begin{proof}
Let $\pi\chi(X)=\lambda$ and $Uc(X)\psi_{\theta}(X)=k$; for each $p\in X$, let $\mathcal{U}_{p}$ be a local $\pi$-base at $p$ such that $|\mathcal{U}_{p}|\leq\lambda$.

Construct an increasing chain $\{A_{\alpha}:\;\alpha<k^{+}\}$ of subsets of $X$ and a sequence $\{\mathcal{U}_{\alpha}:\;0<\alpha<k^{+}\}$ of open collections in $X$ such that:
\begin{enumerate}
\item {$|A_{\alpha}|\leq\lambda^{k}$, $0\leq\alpha<k^{+}$;}
\item{$\mathcal{U}_{\alpha}=\{V\in\mathcal{U}_{p}:\; p\in \bigcup_{\beta<\alpha}A_{\beta}\}$, $0<\alpha<k^{+}$;}
\item{for each $\gamma<k$, if $\mathcal{V}_{\gamma}\in[\mathcal{U}_{\alpha}]^{\leq k}$ and $W=\bigcup_{\gamma<k}cl_{\theta}(\bigcup \overline{\mathcal{V}_{\gamma}})\neq X$, then $A_{\alpha}\setminus W\neq\emptyset$.}
\end{enumerate}

The construction is by trasfinite induction. Let $0<\alpha<k^{+}$ and assume that $\{A_{\beta}:\;\beta<\alpha\}$ has already been constructed. Then $\mathcal{U}_{\alpha}$ is defined by 2., i.e., we put $\mathcal{U}_{\alpha}=\{V:\;\exists p\in\bigcup_{\beta<\alpha}A_{\beta},\;V\in\mathcal{U}_{p}\}$. It follows that $|\mathcal{U}_{\alpha}|\leq\lambda^{k}$. If $\{\mathcal{V}_{\gamma}\}_{\gamma<k}\in[[\mathcal{U}_{\alpha}]^{\leq k}]^{\leq k}$ and $W=\bigcup_{\gamma<k}cl_{\theta}(\bigcup\overline{\mathcal{V}_{\gamma}})\neq X$, then we can choose one point of $X\setminus W$. Let $S_{\alpha}$ be the set of points chosen in this way. Note that $|[[\mathcal{U}_{\alpha}]^{\leq k}]^{\leq k}|\leq\lambda^{k}$. Define $A_{\alpha}$ to be the set $S_{\alpha}\cup(\bigcup_{\beta<\alpha}A_{\beta})$. Then $A_{\alpha}$ satisfies 1., and 3. is also satisfied if $\beta\leq\alpha$. This completes the construction.

Now let $S=\bigcup_{\alpha<k^{+}} A_{\alpha}$; then $|S|\leq k^{+}\lambda^{k}=\lambda^{k}$. The proof is complete if $S=X$. Suppose not and let $p\in X\setminus S$; since $\psi_{\theta}(X)\leq k$, there exist open neighbourhoods $\{U_{\alpha}\}_{\alpha<k}$ of $p$ such that $\{p\}=\bigcap_{\alpha<k}cl_{\theta}(\overline{U_{\alpha}})$. For each $\alpha<k$, let $V_{\alpha}=X\setminus cl_{\theta}(\overline{U_{\alpha}})$.
Then
$S=\bigcup_{\alpha<k}V_{\alpha}\cap S$.
Fix $\alpha<k$. For each $q\in V_{\alpha}\cap S$, there exists $V_{q}\in {\cal U}_q$ such that $\overline{V_{q}}\cap \overline{U_{\alpha}}=\emptyset$ (from the definition of $V_{\alpha}$). We have that $\{V\in\mathcal{U}_{q}:\;V\subseteq V_{q}\}$ is a local $\pi$-base at $q$. Since $q\in \overline{\bigcup\{V\in\mathcal{U}_{q}:\;V\subseteq V_{q}\}}$, we have that $S\cap V_{\alpha}\subseteq\bigcup_{q\in S\cap V_{\alpha}}\overline{\bigcup\{V\in\mathcal{U}_{q}:\;V\subseteq V_{q}\}}\subseteq\overline{\bigcup\{V:\;V\in\mathcal{U}_{q},\;V\subseteq V_{q},\;q\in S\cap V_{\alpha}\}}$. We put $\mathcal{W}_{\alpha}=\{V:\;V\in\mathcal{U}_{q},\;V\subseteq V_{q},\;q\in S\cap V_{\alpha}\}$. 
Since  $Uc(X)\leq k$, by Lemma \ref{lemma} we have that $\forall \alpha<k$ there exists a maximal Urysohn cellular family $\mathcal{W}^{'}_{\alpha}\in[\mathcal{W}_{\alpha}]^{\leq k}$ such that $cl_{\theta}(\bigcup\overline{\mathcal{W}^{'}_{\alpha}})\supseteq\bigcup\mathcal{W}_{\alpha}$. Since $cl_{\theta}(\bigcup\overline{\mathcal{W}^{'}_{\alpha}})$ is closed, it follows that 
$S\cap V_{\alpha}\subseteq \overline{\bigcup\mathcal{W}_{\alpha}}\subseteq cl_{\theta}(\bigcup\overline{\mathcal{W}^{'}_{\alpha}})\subseteq cl_{\theta}(\bigcup_{q\in S\cap V_{q}}\overline{V_{q}})$. 
Then, since $(\bigcup_{q\in S\cap V_{\alpha}}\overline{V_{q}})\cap \overline{U_{\alpha}}=\emptyset$ and $p\notin cl_{\theta}(\bigcup_{q\in S\cap V_{\alpha}}\overline{V_{q}})$, we have that $p\notin cl_{\theta}(\bigcup\overline{\mathcal{W}^{'}_{\alpha}})$. Put $W=\bigcup_{\alpha< k}cl_{\theta}(\bigcup\overline{\mathcal{W}^{'}_{\alpha}})$. Since $|\{V:\;V\in\mathcal{W}^{'}_{\alpha}$ for some $\alpha<k\}|\leq kk=k<k^{+}$, there is an $\alpha_{0}<k^{+}$ such that $\mathcal{W}^{'}_{\alpha}\in[\mathcal{U}_{\alpha_{0}}]^{\leq k}$ for each $\alpha< k$. Hence, by 3., one has $A_{\alpha_{0}}\setminus W\neq\emptyset$. But $W\supseteq\bigcup_{\alpha<k}(V_{\alpha}\cap S)=S$ and $A_{\alpha_{0}}\setminus W\subseteq S\setminus W=\emptyset$; a contradiction.
\end{proof}

\begin{cor}\cite{SAP}
Let $X$ be a regular space. Then $|X|\leq\pi\chi(X)^{c(X)\psi(X)}$.
\end{cor}

\section{Variations of the Dissanayake and Willard's inequality
$|X|\leq 2^{aL_{c}(X)\chi(X)}$ and of the de Groot's inequality $|X|\leq 2^{hL(X)}$ in the class of $T_1$ spaces.}\label{Section2}

In Proposition \ref{prop1} it was shown that  Urysohn axiom is equivalent to $\{x\}=\bigcap\{cl_{\theta}(\overline{U}):\;U$ open, $x\in U\}$, for every point $x$ of the space. The following example shows that in spaces which are not Urysohn the previous intersection can be large.

\begin{es}\rm
Any infinite space $X$ with the cofinite topology is a $T_{1}$, not Hausdorff space for which there is a point $x$ such that $\bigcap\{cl_{\theta}(\overline{U}):\;x\in U\}$ has large cardinality.
\end{es}

The example above gives a motivation to introduce the following definition:
\begin{defin}\rm
Let $X$ be a $T_{1}$ topological space and for all $x\in X$, let
$$Uw(x)=\bigcap\{cl_{\theta}(\overline{U}):\;x\in U,\;U\;open\}.$$

The \textit{Urysohn width} is:

$$UW(X)=sup\{|Uw(x)|:\;x\in X\}.$$
\end{defin}

It is clear that if $X$ is a Urysohn space then $UW(X)=1$.

Recall that $HW(X)=sup\{|Hw(x)|:\;x\in X\}$ is the \textit{Hausdorff width}, where $Hw(x)=\bigcap\{\overline{U}:\;x\in U,\;U\;open\}$ \cite{BSS}.
Since the $\theta$-closure of a set contains its closure we have that $HW(X)\leq UW(X)$.
\begin{quest}\rm
Is $HW(X)=UW(X)$ in some class of non regular spaces?
\end{quest}

\begin{defin}\rm\cite{BSS}
Let $X$ be a space and $x\in X$. $$\psi w(x)= min\{|\mathcal{U}_{x}|:\;\bigcap\{\overline{U}:\;U\in\mathcal{U}_{x}\}=Hw(x),\;\mathcal{U}_{x}\; is\; a\;\\$$ $$family\; of\; open\; neighborhood\; of \;x\};$$
and
$$\psi w(X)= sup\{\psi w(x):\;x\in X\}.$$
\end{defin}

Similarly, we introduce the following definition.

\begin{defin}\rm Let $X$ be a space and $x\in X$.
$$\psi w_{\theta}(x)= min\{|\mathcal{U}_{x}|:\;\bigcap\{cl_{\theta}(\overline{U}):\;U\in\mathcal{U}_{x}\}=Uw(x),\;\mathcal{U}_{x}\; is\; a\;\\$$ $$family\; of\; open\; neighborhood\; of \;x\};$$
and
$$\psi w_{\theta}(X)= sup\{\psi w_{\theta}(x):\;x\in X\}.$$
\end{defin}

Of course, if $X$ is a $T_{1}$ space then $\psi w(X)\leq\psi w_{\theta}(X)\leq\chi(X)$; further if $X$ is a Urysohn space then we have that $\psi w_{\theta}(X)=\psi_{\theta}(X)$.

We introduce the following definition:

\begin{defin}\rm Let $Y$ be a subset of a space $X$.

The $\theta$\textit{-almost Lindel\"of degree} of a subset $Y$ of a space $X$ is

$\theta\hbox{-}aL(Y,X)=\min \{k : $ for every cover $\mathcal{V}$ of $Y$ consisting of open subsets of  $X$,
there exists $\mathcal{V'}\subseteq\mathcal{V}$  such that $|\mathcal{V'}|\leq k$ and $\bigcup\{cl_{\theta}(\overline{V}):\;V\in\mathcal{V'}\}=Y\}.$

The function $\theta$-$aL(X,X)$ is called $\theta$\textit{-almost Lindel\"of degree of the space} $X$ and denoted by $\theta$-$aL(X)$.

The $\theta$\textit{-almost  Lindel\"of degree with respect to closed subsets of }$X$, denoted by $\theta$-$aL_{c}(X)$, is the cardinal $sup\{\theta$-$ aL(C,X):\;C\subseteq X\;is\;closed\}$.

The $\theta$\textit{-almost  Lindel\"of degree with respect to $\theta$-closed subsets of }$X$, denoted by $\theta$-$aL_{\theta}(X)$, is the cardinal $sup\{\theta$-$ aL(B,X):\;B\subseteq X\;is\;\theta\hbox{-}closed\}$.
\end{defin}

Of course $\theta$-$aL(X)\leq aL(X)$, for every space $X$.
Using a slight modification of Example 2.3 in \cite{BGW} we prove that the previous inequality can be strict.

\begin{es}\rm A space $X$ such that $\theta$-$aL(X)< aL(X)$. 

 Let $k$ be any uncountable cardinal, let $\mathbb{Q}$ be the set of all the rationals and let $\mathbb{P}$ be the set of the irrationals. Put $X=(\mathbb{Q}\times k)\cup \mathbb{P}$. We topologized $X$ as follows. If $q\in\mathbb{Q}$ and $\alpha<k$ then a neighborhood base at $(q,\alpha)$ is $\mathcal{U}(q,\alpha)=\{U_{n}(q,\alpha):\;n\in\omega\}$ where $$U_{n}(q,\alpha)=\{(r,\alpha):\;r\in\mathbb{Q}\;\hbox{and}\;|r\hbox{-}q|<\frac{1}{n}\}.$$
If $p\in\mathbb{P}$ a neighborhood base at $p$ takes the form:
$$\{\{b\in\mathbb{P}:\;|b\hbox{-}p|<\frac{1}{n}\}\cup\{(q,\alpha):\;\alpha<k\;\hbox{and}\;|q\hbox{-}p|<\frac{1}{n}\}:\;n\in\omega\}.$$

For every $q\in\mathbb{Q}$, $\alpha<k$ and $n\in\omega$ we have that:
$$\overline{U_{n}(q,\alpha)}=U_{n}(q,\alpha)\bigcup\{(r,\alpha):\;r\in\mathbb{Q},\;|r\hbox{-}q|<\frac{1}{n}\}\bigcup\{p\in\mathbb{P}:\;|q\hbox{-}p|<\frac{1}{n}\};$$
and:
$$cl_{\theta}(\overline{U_{n}(q,\alpha)})=\overline{U_{n}(q,\alpha)}\bigcup\{(r,\beta):\;|r-q|<\frac{1}{n},\;\beta<k\;\hbox{and}\;\beta\neq\alpha\}.$$

Let $\alpha<k$, we have that $X=\bigcup_{q\in\mathbb{Q}} cl_{\theta}(\overline{\mathcal{U}(q,\alpha)})$ and so $\theta\hbox{-}aL(X)=\aleph_{0}$ but we have that $aL(X)=2^{\aleph_{0}}$.

\end{es}

It is easy to show that the almost Lindel\"of degree is hereditary with respect to $\theta$-closed subsets. It is natural to ask:

\begin{quest}\rm
Is the $\theta$-almost Lindel\"of degree hereditary with respect to $\theta$-closed subsets?
\end{quest}

We find out (Proposition \ref{prop2}) that the $\theta$-almost Lindel\"of degree is hereditary with respect to a new class of spaces that we call $\gamma$\textit{-closed}.
\begin{defin}\rm
Let $X$ be a topological space and $A\subseteq X$. The $\gamma$\textit{-closure} of the set $A$ is

$cl_{\gamma}(A)=\{x:$ for every open neighborhood of $X,\; cl_{\theta}(\overline{U})\cap A\neq\emptyset\}.$

\noindent$A$ is said to be $\gamma$-closed if $A = cl_\gamma(A)$.
\end{defin}

The following example shows that the $\gamma$-closure and the $\theta$-closure of a subset of a topological space can be different.
\begin{es}\label {E}\rm
A Urysohn space $X$ having a subset $Y$ such that $cl_{\gamma}(Y)\neq cl_{\theta}(Y)$.
\end{es}
\proof
Let $\mathbb{R}= A \cup B \cup C \cup D$ where $A, B, C,  D$ are pairwise disjoint and each is dense in $\mathbb{R}$.  Let $A'$ be a topological copy of $A$; points in $A'$ are denoted as $a'$ where $a \in A$.  

Let $a,b \in \mathbb{R}$.  A base for $X$ is generated by these families of open sets:  \newline 
(1)$\{(a,b) \cap A: a, b \in \mathbb{R}, a < b\}$ \newline 
(2)$\{(a,b)\cap C: a, b \in \mathbb{R}, a < b\}$ , \newline 
(3)$\{(a,b) \cap A': a, b \in \mathbb{R}, a < b\}$,  \newline  
(4)$\{(a,b) \cap (A \cup B \cup C):  a, b \in \mathbb{R}, a < b\}$, and
 \newline 
(5)$\{ (a,b) \cap(C \cup D \cup A'):  a, b \in \mathbb{R}, a < b\}$.

Note that for every $a,\;b\in\mathbb{R}$, $\overline{(a,b) \cap A} = [a,b] \cap (A \cup B)$, $\overline{(a,b) \cap A'} = [a,b] \cap (A' \cup D)$, $\overline{(a,b) \cap C} = [a,b] \cap (B \cup C \cup D)$, $cl_{\theta}(\overline{(a,b) \cap A}) = [a,b] \cap (A \cup B \cup C )$ and  $cl_{\theta}(\overline{(a,b) \cap A'}) = [a,b] \cap (A' \cup D \cup C )$. For these reasons we can say that if $a,\;b\in\mathbb{R}$ and if we put $Y=(a,b)\cap C$, we have that $cl_{\theta}(Y)=[a,b]\cap(B\cup C\cup D)$ and $cl_{\gamma}(Y)=[a,b]\cap(A\cup B\cup C\cup D\cup A')$.
\endproof

We have the following:
\begin{prop}\label{prop2}
The $\theta$-almost Lindel\"of degree is hereditary with respect to $\gamma$-closed subsets.
\end{prop}
\proof
Let $X$ be a topological space such that $\theta\hbox{-}aL(X)\leq k$ and let $C\subseteq X$ be $\gamma$-closed set. $\forall x\in X\setminus C$ we have that there exists an open neighborhood $U_{x}$ of $x$ such that $cl_{\theta}(\overline{U})\subseteq X\setminus C$. Let $\mathcal{U}$ be a cover of $C$ consisting of open subsets of $X$. Then $\mathcal{V}=\mathcal{U}\bigcup\{U_{x}:\;x\in X\setminus C\}$ is an open cover of $X$ and since $\theta\hbox{-}aL(X)\leq k$, there exists
$\mathcal{V'}\in[\mathcal{V}]^{\leq k}$ such that $X=\bigcup \{cl_{\theta}(\overline{V}): V\in {\mathcal{V'}}\}$. Then there exists
$\mathcal{V''}\in[\mathcal{U}]^{\leq k}$ such that $C\subseteq \bigcup \{cl_{\theta}(\overline{V}): V\in {\mathcal{V''}}\}$; this proves that
$\theta\hbox{-}aL(C)\leq k$.
\endproof

Now we use $UW(X)$ and $\theta\hbox{-}aL_{\theta}(X)$ to restate Theorem 2.22 in \cite{BSS} in the class of $n$-Uysohn spaces. The proof follows step by step the proof of Theorem 2.22 in \cite{BSS}.

\begin{theorem} 
If $X$ is a $T_{1}$ $n$-Urysohn ($n\in\omega$)  space, then $|X|\leq UW(X)2^{\theta\hbox{-}aL_{\theta}(X)\chi(X)}$.
\end{theorem}
\proof
Let $UW(X)\leq k$, $\theta$-$ aL_{\theta}(X)\chi(X)\leq\tau$. For all $x\in X$, let $\mathcal{U}_{x}$ be a local base and $|\mathcal{U}_{x}|\leq\tau$. Note that for all $x\in X$, $Uw(x)=\bigcap\{cl_{\theta}(\overline{U}):\;U\in\mathcal{U}_{x}\}$. Construct $\{H_{\alpha}:\;\alpha\in\tau^{+}\}$ and $\{\mathcal{B}_{\alpha}:\;\alpha\in\tau^{+}\}$ such that:
\begin{enumerate}
\item $H_{\alpha}\subset H_{\beta}\subset X$, for all $\alpha,\;\beta\in\tau^{+}$;
\item $H_{\alpha}$ is $\theta$-closed for all $\alpha\in\tau^{+}$;
\item $|H_{\alpha}|\leq 2^{\tau}$ for all $\alpha\in\tau^{+}$;
\item if $\{H_{\beta}:\;\beta\in\alpha\}$ are defined for some $\alpha\in\tau^{+}$, then $\mathcal{B}_{\alpha}=\bigcup\{\mathcal{U}_{x}:\;x\in\bigcup\{H_{\beta}:\;\beta\in\alpha\}\}$;
\item if $\alpha\in\tau^{+}$ and $\mathcal{W}\in[\mathcal{B}_{\alpha}]^{\leq\tau}$ is such that $X\setminus(\bigcup\{cl_{\theta}(\overline{U}):\;U\in\mathcal{W}\})\neq\emptyset$ then $H_{\alpha}\setminus(\bigcup\{cl_{\theta}(\overline{U}):\;U\in\mathcal{W}\})\neq\emptyset$.
\end{enumerate}

Let $\alpha\in\tau^{+}$ and $\{H_{\beta}:\beta\in\alpha\}$ be already defined. For all $\mathcal{W}$ as in 5., choose a point $x(\mathcal{W})\in X\setminus(\bigcup\{cl_{\theta}(\overline{U}):\;U\in\mathcal{W}\})$ and let $C_{\alpha}$ be the set of these points. Let $H_{\alpha}=[\bigcup\{H_{\beta}:\;\beta\in\alpha\}\cup C_{\alpha}]_{\theta}$. Considering the fact that if $X$ is a $n$-Urysohn space we have that for every $A\subseteq X$, $|[A]_{\theta}|\leq|A|^{\chi(X)}$ \cite{BCM} we have that $|H_{\alpha}|\leq2^{\tau}$. Let $H=\bigcup\{H_{\beta}:\;\beta\in\tau^{+}\}$. Since $t_{\theta}(X)\leq\chi(X)\leq\tau$, $\tau^{+}$ is regular and $\{H_{\alpha}:\;\alpha\in\tau^{+}\}$ is an increasing family of my{$\theta$-}closed sets of lenght $\tau^{+}$, we have that $H$ is $\theta$-closed. Also $|H|\leq 2^{\tau}$. Let $H^{*}=\bigcup\{Uw(x):\;x\in H\}\supseteq H$. Then $|H^{*}|\leq k2^{\tau}$.

We want to prove that $X=H^{*}$. Suppose that there exists a point $q\in X\setminus H^{*}\subset X\setminus H$. Then for all $x\in H$ there is $U(x)\in\mathcal{U}_{x}$ such that $q\notin cl_{\theta}(\overline{U(x)})$. From $\theta$-$aL_{\theta}(X)\leq \tau$ choose $H^{'}\in[H]^{\leq\tau}$ such that $H\subseteq \bigcup\{cl_{\theta}(\overline{U(x)}):\;x\in H^{'}\}$. Then $H^{'}\subseteq H_{\alpha}$ for some $\alpha\in\tau^{+}$ and hence $\mathcal{W}=\{cl_{\theta}(\overline{U(x)}):x\in H^{'}\}$ $\in[\mathcal{B}_{\alpha +1}]^{\leq\tau}$ and $q\in X\setminus(\bigcup\{cl_{\theta}(\overline{U}):\;U\in\mathcal{W}\})\neq\emptyset$. Hence we have already chosen $x(\mathcal{W})\in H_{\alpha}+\}\cap(H\setminus\bigcup\{cl_{\theta}(\overline{U(x)}):\;x\in H^{'}\})\subseteq H\cap(X\setminus H)$ a contradiction. Hence $X=H^{*}$ and $|X|\leq k2^{\tau}$.
\endproof

Now we use $UW(X)$,$\psi w_{\theta}(X)$  and $h\theta\hbox{-}aL(X)$ to present a variation of the Theorem 2.23 in \cite{BSS}. The proof of Theorem~\ref{1} follows step by step the proof of Theorem 2.23 in \cite{BSS}.

\begin{theorem} \label{1}
If $X$ is a $T_{1}$ space then $|X|\leq UW(X)\psi w_{\theta}(X)^{h\theta\hbox{-}aL(X)}$.
\end{theorem}
\proof
Let $UW(X)\leq k$, $h\theta$-$aL(X)\leq\tau$ and $\psi w_{\theta}(X)\leq\lambda$. For all $x\in X$, let $\mathcal{U}_{x}$ be a family of open neighborhood of $x$ such that $|\mathcal{U}_{x}|\leq\lambda$ and  $Uw(x)=\bigcap\{cl_{\theta}(\overline{U}):\;U\in\mathcal{U}_{x}\}$. By trasfinite induction we construct two families $\{H_{\alpha}:\;\alpha\in\tau^{+}\}$ and $\{\mathcal{B}_{\alpha}:\;\alpha\in\tau^{+}\}$ such that:
\begin{enumerate}
\item $\{H_{\alpha}:\;\alpha\in\tau^{+}\}$ is an increasing sequence of subsets of $X$;

\item $|H_{\alpha}|\leq k\lambda^{\tau}$ for all $\alpha\in\tau^{+}$;
\item if $\{H_{\beta}:\;\beta\in\alpha\}$ are defined for some $\alpha\in\tau^{+}$, then $\mathcal{B}_{\alpha}=\bigcup\{\mathcal{U}_{x}:\;x\in\bigcup\{Uw(y):\;y\in\bigcup\{H_{\beta}:\;\beta\in\alpha\}\}\}$;
\item if $\alpha\in\tau^{+}$ and $\mathcal{W}\in[\mathcal{B}_{\alpha}]^{\leq\tau}$ is such that $X\setminus(\bigcup\{cl_{\theta}(\overline{U}):\;U\in\mathcal{W}\})\neq\emptyset$ then $H_{\alpha}-(\bigcup\{cl_{\theta}(\overline{U}):\;U\in\mathcal{W}\})\neq\emptyset$.
\end{enumerate}
Let $\alpha\in\tau^{+}$ and $\{H_{\beta}:\beta\in\alpha\}$ be already defined. For all $\mathcal{W}$ as in 4., choose a point $x(\mathcal{W})\in X\setminus(\bigcup\{cl_{\theta}(\overline{U}):\;U\in\mathcal{W}\})$ and let $C_{\alpha}$ be the set of these points.

Let $H_{\alpha}=\bigcup\{H_{\beta}:\;\beta\in\alpha\}\cup C_{\alpha}$. Then $|H_{\alpha}\leq k\lambda^{\tau}$.

Let $H=\bigcup\{H_{\alpha}:\;\alpha\in\tau^{+}\}$ and $H^{*}=\bigcup\{Uw(x):\;x\in H\}\supseteq H$. Then $|H^{*}|\leq k\lambda^{\tau}$.

We want to prove that $X=H^{*}$. Suppose that there exists a point $q\in X\setminus H^{*}$. Then $q\notin Uw(x)$, $\forall x\in H$. Hence for all $x\in H$ there is $U(x)\in\mathcal{U}_{x}$ such that $q\notin cl_{\theta}(\overline{U(x)})$. From $h\theta$-$aL(X)\leq \tau$ choose $H^{'}\in[H]^{\leq\tau}$ such that $H\subseteq \bigcup\{cl_{\theta}(\overline{U(x)}):\;x\in H^{'}\}$. Let $\mathcal{W}=\{\overline{U(x)}:x\in H^{'}\}$. We have that $H^{'}\subseteq H_{\alpha}$ for some $\alpha\in\tau^{+}$ and $\mathcal{W}\in[\mathcal{B}_{\alpha +1}]^{\leq\tau}$ and $X\setminus(\bigcup\{cl_{\theta}(\overline{U}):\;U\in\mathcal{W}\})\neq\emptyset$. Hence we have already chosen $x(\mathcal{W})\in X\setminus(\bigcup\{cl_{\theta}(\overline{U}):\;U\in\mathcal{W}\}\subseteq X\setminus H$ and $x(\mathcal{W})\in H$ a contradiction. Hence $X=H^{*}$ and $|X|\leq k\lambda^{\tau}$.
\endproof

\begin{cor}
If $X$ is a Urysohn space then $|X|\leq \psi_{\theta}(X)^{h\theta\hbox{-}aL(X)}$.
\end{cor}

\section{The Urysohn point separating weight}\label{section3}

\begin{defin}\rm\cite{BCCS}
A \textit{Hausdorff point separating open cover} $\mathcal{S}$ for a space $X$ is an open cover of $X$ having the property that for each distinct points $x,\;y\in X$ there exists $S\in\mathcal{S}$ such that $x\in S$ and $y\notin \overline{S}$. 

The \textit{Hausdorff point separating weight of a space} $X$  is 

$Hpsw(X)=min\{\tau:\;X\;\hbox{has a Hausdorff point separating open cover}\;\mathcal{S}\;$ $\hbox{such that each point of}\;X\;\hbox{is contained in at most}\;\tau\;\hbox{elements of}\;\mathcal{S}\}.$
\end{defin}

Following the same idea as in \cite{BCCS} we introduce the following definition:
\begin{defin}\rm
A \textit{Urysohn point separating open cover} $\mathcal{S}$ for a space $X$ is an open cover of $X$ having the property that for each distinct points $x,\;y\in X$ there exists $S\in\mathcal{S}$ such that $x\in S$ and $y\notin cl_{\theta}(\overline{S})$.
\end{defin}

\begin{defin}\rm
The \textit{Urysohn point separating weight} of a Urysohn space $X$ is the cardinal:
$$Upsw(X)=min\{\tau:\;X\;\hbox{has a Urysohn point separating open cover}\;\mathcal{S}\;\\$$ $$\hbox{such that each point of X is contained in at most}\;\tau\hbox{elements of}\;\mathcal{S}\}+\aleph_{0}.$$
\end{defin}

Note that $Hpsw(X)\leq Upsw(X)$, for every Urysohn space $X$.

\noindent The proof of the following theorem follows step by step the proof of Theorem 20 in \cite{BCCS}.
\begin{theorem}
If $X$ is a Urysohn space then $nw(X)\leq Upsw(X)^{\theta\hbox{-}aL_{c}(X)}$.
\end{theorem}
\proof
Let $\theta$-$aL_{c}(X)=k$ and $\mathcal{S}$ a Urysohn point separating open cover for $X$ such that for each $x\in X$, $|\mathcal{S}_{x}|\leq\lambda$, where $\mathcal{S}_{x}$ is the collection of members of $\mathcal{S}$ containing $x$.

We first show that $d(X)\leq \lambda^{k}$. $\forall$ $\alpha< k$ construct a subset $D_{\alpha}$ of $X$ such that:
\begin{enumerate}
\item $|D_{\alpha}|\leq \lambda^{k}$;
\item if $\mathcal{U}$ is a subcollection of $\bigcup\{\mathcal{S}_{x}:\;x\in \bigcup_{\beta<\alpha} D_{\beta}\}$ such that $|\mathcal{U}|\leq k$ and if $X\setminus\bigcup cl_{\theta}(\overline{\mathcal{U}})\neq\emptyset$ we have that $D_{\alpha}\setminus\bigcup cl_{\theta}(\overline{\mathcal{U}})\neq\emptyset$.
\end{enumerate}
Such a $D_{\alpha}$ can be constructed since the member of possible $\mathcal{U}'s$ at the $\alpha th$ stage of construction is $\leq(\lambda^{k} k\lambda)^{k}=\lambda^{k}$.

Let $D=\bigcup_{\alpha<k^{+}} D_{\alpha}$. We have that $|D|\leq \lambda^{k}$. We want to prove that $\overline{D}=X$. Suppose that there exists $p\in X\setminus\overline{D}$, since $Upsw(X)\leq \lambda$, $\forall x\in\overline{D}$, there exists $V_{x}\in\mathcal{S}_{x}$: $p\notin cl_{\theta}(\overline{V_{x}})$. Since $x\in\overline{D}$, $V_{x}\cap D\neq\emptyset$. Let $y\in V_{x}\cap D$, so $V_{x}\in\bigcup\{\mathcal{S}_{y}:\;y\in D\}$. Put $\mathcal{W}=\{V_{x}:\;x\in\overline{D}\}\subseteq\bigcup\{\mathcal{S}_{y}:\;y\in D\}$. $\mathcal{W}$ is an open cover of $\overline{D}$ and since $\theta$-$aL_{c}(X)\leq k$, there exists $\mathcal{W'}\subseteq\mathcal{W}$ with $|\mathcal{W'}|\leq k$ such that $\overline{D}\subseteq\bigcup\{cl_{\theta}(\overline{V}):\;V\in\mathcal{W'}\}$ and $p\notin \bigcup\{cl_{\theta}(\overline{V}):\;V\in\mathcal{W'}\}$ and this contradicts 2..

Since $d(X)\leq\lambda^{k}$ we have that $|\mathcal{S}|\leq\lambda^{k}$.

Let $\mathcal{N}=\{X\setminus S:\;S\;is\;the\;union\;of\;at\;most\;k\;members\;of\;\mathcal{S}\}$. $|\mathcal{N}|\leq\lambda^{k}$ and $\mathcal{N}$ is a network for $X$.

\endproof

\begin{theorem}
If $X$ is a Urysohn space then $|X|\leq Upsw(X)^{\theta\hbox{-}aL_{c}(X)\psi(X)}$.
\end{theorem} 
\proof
If $X$ is a $T_{1}$ space then $|X|\leq nw(X)^{\psi(X)}$ and using the theorem above we have that $|X|\leq nw(X)^{\psi(X)}\leq Upsw(X)^{\theta\hbox{-}aL_{c}(X)\psi (X)}$.
\endproof

\section*{Acknowledgement}
The authors are very grateful to J. Porter for suggesting Example \ref{E}.

\end{document}